\documentclass{amsart}
\usepackage{graphicx}
\usepackage{amssymb}
\newtheorem{thm}{Theorem}[section]
\newtheorem{cor}[thm]{Corollary}
\newtheorem{lemma}[thm]{Lemma}
\newtheorem{ex}[thm]{Example}

\newtheorem{prop}[thm]{Proposition}
\theoremstyle{definition}
\newtheorem{defn}[thm]{Definition}
\newtheorem{remark}[thm]{Remark}
\theoremstyle{question}

\theoremstyle{Conjecture}
\newtheorem{con}[thm]{Conjecture}
\theoremstyle{Problem}

\numberwithin{equation}{section}

\begin{document}

\title[ the same number of centralizers]{ On the conjecture of groups with  the same number of centralizers}%
\author{N. Ahmadkhah and M. Zarrin}%

\address{Department of Mathematics, University of Kurdistan, P.O. Box: 416, Sanandaj, Iran}%
 \email{N.Ahmadkhah20@gmail.com}
 \address{Department of Mathematics, Texas State University, 601 University Drive, San Marcos, TX, 78666, USA}
 \email{m.zarrin@txstate.edu}
\begin{abstract}
For any group G, let $cent(G)$ denote the set of all centralizers of $G$. The authors in \cite{KZ}, Groups with the same number of centralizers,
 J. Algebra Appl. (2021) 2150012 (6 pages), posed  the following conjecture: 	Let $G$ and $S$ be finite groups. 
 Is it true that if $|Cent(G)|=|Cent(S)|$ and $|G'|=|S'|$, then $G$ is isoclonic to $S$? In this paper, among other things, we give a negative answer to this conjecture.

\end{abstract}
\maketitle

\section{\textbf{ Introduction}}

Throughout this paper all groups mentioned are assumed to be finite and we will use usual notation; for example, for a group $G$ and its subgroup $H$, we define sets $Cent(G)=\{C_G(x) \mid x\in G\}$ and $Cent_G(H)=\{C_G(h)\mid h\in H\}$, where $C_G(h)=\{g\in G \mid gh = hg\}$ is the centralizer of $h$ in $G$ and we say $G$ is an $\mathcal C_n$-group if $|Cent(G)|=n$. Since for every $x\in G$, $C_G(x)\cap H=C_H(x)$, so $|Cent(H)| \le |Cent_G(H)|\le |Cent(G)|$. It is clear that $G$ is an $\mathcal C_1$-group if and only if it is abelian. Belcastro and Sherman in \cite{BS} proved that there is no n-centralizer groups for $n=2,3$ and they characterized
finite $\mathcal C_n$-groups, for $n=4,5$. Other authors investigated the $\mathcal C_n$-groups, for $n =6, 7, 8, 9, 10$. (See for instance \cite{{AJH},{A}, {JMR}, {JMR1}, {Z1}, {Z4}} ).  
In \cite{AT}, Ashrafi and Taeri, studied the influence of $|Cent(G)|$ on the structure of a group $G$ and they raised the following question:\\
 Let G and H be finite simple groups. Is it true that if $|Cent(G)|=|Cent(H)|$, then $G$ is isomorphic to $H$? Zarrin in \cite{Z1} disproved their question with a counterexample and
  he in \cite{Z2} showed that for every two isoclinic groups $G$ and $S$, $|Cent(G)|=|Cent(S)|$ but the conversely is not true (see \cite{Z1}) (note that two simple groups are isomorphic if and only if they are isoclinic.).\\
 In what follows, the authors in \cite{KZ} raised the following Conjecture:
  
  \begin{con}\label{c}
  	Let $G$ and $S$ be finite groups. Is it true that if $|Cent(G)|=|Cent(S)|$ and $|G'|=|S'|$, then $G$ is isoclonic to $S$?
  \end{con}
  
In this paper, we give a negative answer to this Conjecture. To disprove the above Conjecture, we need to define the new class of groups. 

 \begin{defn}
A group $G$ is called a $Cpo$-group if  the order of the centralizer of all non-central elements of $G$ is prime.
\end{defn}

In Section $2$, we characterize $Cpo$-groups and we proved that they are of order $p$ or $pq$ for some prime numbers $p, q$, see Theorem \ref{cl5}.\\ 

A group $G$ is called an $CA$-group, if the centralizer of every non-central element of $G$ is abelian. Clearly the $Cpo$-groups are $CA$-groups.

In Section $3$, we obtain $|Cent(G)|$ for $Cpo$-group $G$, see Theorem \ref{cl}, and we give some counterexample for Conjecture \ref{c} among the class of $Cpo$-groups.\\

\section{\textbf{ Characterization of $Cpo$-groups }}

Let $G$ be a finite group. In this paper, $G$ is called a $Cpo$-group if the order of the centralizer of all non-trivial elements of $G$ is prime.
 At first we give some example of $Cpo$-groups and then characterize $Cpo$-groups.

\begin{ex}
 The groups of prime order are $Cpo$-groups. Also the symmetric group of degree $3$, $S_3$ and the dihedral groups of degree $10$, $D_{10}$ are $Cpo$-groups. 
\end{ex}

It is clear that, for a $Cpo$-group $G$, $C_G(x)=<x>$ for all non-trivial elements $x\in G$, so every non-trivial element of $G$ is of prime order. 
Conversely, it is not necessarily true, For instance the alternating group of degree $5$, $A_5$.\\
Note that a non-abelian $Cpo$-group is centerless group, means $Z(G)=1$.

\begin{lemma}\label{cl1}
Let $G$ be a $Cpo$-group and $H$ be a non-centerless subgroup of $G$. Then $|H|$ is prime.
\end{lemma}

\begin{proof}
Straight forward. 
\end{proof}

\begin{cor}\label{cl2}If $G$ is a $Cpo$-group, then:
\begin{itemize}

\item [(1)]For every $p\in \pi(G)$ the order of $p$-sylow subgroup of $G$ is $p$. In particular $|G|=\sqcap p_i$ where $p_i\in \pi(G)$.
\item [(2)]For every non-central element $x\in G$, $C_G(x)$  is a sylow subgroup of $G$.
\item [(3)]$C_G(x)=C_G(y)$ or $C_G(x)\cap C_G(y)=1$ for every non-central element $x, y \in G$.
\end{itemize}
\end{cor}
\begin{lemma}\label{cl3}
Let $G$ be a $Cpo$-group. Then:
\begin{itemize}
\item [(1)]Any proper subgroup of $G$ is nilpotent. In particular $G$ is solvable.
\item [(2)]$|G|$ is prime or $G$ is non-nilpotent with $|G|=p^a q^b$ where $p$ and $q$ are unequal prime numbers and $m$, $n\in\mathbb N$.  
\end{itemize}
\end{lemma}

\begin{proof}
(1)~~Assume that $H$ is a proper subgroup of $G$, so by Corollary \ref{cl2}, $\pi(G)\backslash \pi(H)\neq \varnothing$ and since for each $x\in G$, 
the subgroup $<x>\le C_G(x)$, therefore all  non-trivial elements of a $Cpo$-group have prime order, thus by 2.10 of \cite{D}, the result follows.\\

(2)~~ For a non-trivial finite group $G$, we have $|\pi(G)|=1$ or $|\pi(G)|>1$. If $|\pi(G)|=1$, then $Cpo$-group $G$ must have prime order.
 In the case that $|\pi(G)|>1$, since for each  non-trivial element $x\in G$ the order of $C_G(x)$ is prime and $|Z(G)|$ divides $|C_G(x)|$,
  hence $Z(G)=1$ and $G$ is non-nilpotent, therefore by Part (1) and Schmidt's Theorem (9.1.9, \cite{R}), the result follows.\\

\end{proof}

\begin{lemma}
If $G$ is a $Cpo$-group and $1\neq H\leq G$ and $1\neq N\unlhd G$, then $H$  and $\frac{G}{N}$ are $Cpo$-groups.
\end{lemma}
\begin{cor}\label{cl4}
If $G$ is a $Cpo$-group, then there exists  $p\in \pi(G)$  such that the fitting subgroup $F(G)=O_p(G)$ in which $O_p(G)$ is the largest normal $p$-subgroup of $G$, and $O_{p'}(G)=1$.
\end{cor}
\begin{proof}
As $F(G)\neq1$ and by Lemma \ref{cl1}, $|F(G)|=p$ for some $p\in \pi(G)$ and therefore $F(G)$ is a $p$-sylow subgroup of $G$ by Part (1) of Corollary \ref{cl2}. 
It follows that $F(G)=O_p(G)$ and $O_{p'}(G)=1$.  
\end{proof}
In the following we have the main theorem.

\begin{thm}\label{cl5}
Let $G$ be a $Cpo$-group. Then $G$ is one of the following groups.
\begin{itemize}
\item [(1)]$|G|=p$ where $p$ is prime number.
\item [(2)]$|G|=pq$, where $p$ , $q$ are  distinct primes  and $F(G)=G'$, $|F(G)|=p$  such that $p> q$ and $q|p-1$.
\end{itemize}
\end{thm}

\begin{proof}
Suppose that $|G|$ is not prime. Then by Part (1) of Corollary \ref{cl2} and Part (1) of Lemma \ref{cl3}, $|G|=pq$  in which $p\neq q$. Now since $F(G)$ is a sylow subgroup of $G$ by corollary \ref{cl4}, we set $F(G)=O_p(G)\in syl_p(G)$ and $Q \in syl_q(G)$ where $|Q|=q$. If $p< q$, then $[G:Q]=p$, so $Q$ is a normal subgroup of $G$. On the other hand  by Lemma \ref{cl3}, $G$ is solvable, it follows that $Q\cap F\neq1$, a contradiction. Therefore $p> q$ and  $F(G)=G'$. Besides $G$ is not abelian, so $q | p-1$
\end{proof}

\begin{cor}
Let $G$ be a $Cpo$-group. Then $G$ is a Frobenius group with Frobenius kernel $P$ and Frobenius complement $H$, such that $P$ is a normal 
Sylow $p$-subgroup of $G$ of order $p$ for some prime $p$ dividing the order of $G$, $H$ is an abelian $p'$-subgroup of $G$.
\end{cor}

\begin{proof}
By taking $F(G)=P$ in Part (2) of Theorem \ref{cl5} and classification of finite $CA$-groups by Schmidt \cite{S}, the result follows.
\end{proof}

\section{\textbf{Groups with the Same number of centralizers }}

 Let $G$ and $H$ are two finite groups with the same number of centralizers.\\
 If $|G|=|H|$, $|Z(G)|=|Z(H)|$ and the order of elements $Cent(G)$ and $Cent(H)$  are equal, then it is not necessarily true that $G\cong H$. For example $Q_8$ and $D_8$.\\

In this section we consider groups with the same number of centralizers and we give some counterexample for Conjecture \ref{c}.\\

The groups $G$ and $H$ are said to be isoclinic if there are two isomorphisms
$\varphi:\frac{G}{Z(G)}\rightarrow \frac{H}{Z(H)}$ and $\phi:G'\rightarrow H'$ such that if  
$$\varphi(g_1Z(G))= h_1Z(H) ~~and~~ \varphi(g_2Z(G))= h_2Z(H),$$
with $g_1, g_2 \in G$, $h_1,  h_2 \in H$, then
$$\phi([g_1, g_2])=[h_1, h_2].$$ 
  The pair $(\varphi, \phi)$ is called a isoclinism from $G$ to $H$ and we write $G\sim H$.\\

Now we prove that the number of centralizers of a non-abelian $Cpo$-group is focused on the largest prime divisor of group.

\begin{thm}\label{cl}
Let $G$ be a non-abelian $Cpo$-group and $p$ be the largest prime divisor of $|G|$. Then $|Cent(G)|=p+2$.
\end{thm}

\begin{proof}
By Applying Part (2) of Theorem \ref{cl5}, assume that $|G|=pq$ for some distinct prime numbers $p, q$, where $p>q$. Then $G$ has only one element $x$ such that $|C_G(x)|=p$ by Part (2) of Corollary \ref{cl2} and Part (2) of Theorem \ref{cl5}. It follows that the centralizers of $pq-p$ other elements  of $G$ have order $q$. Therefore by Part (3) of Corollary \ref{cl2}, we imply that $G$ has $\frac{pq-p}{q-1}=p$ centralizers of order $q$ and since $|Z(G)|=1$, hence $|Cent(G)|=p+2$.  
\end{proof}

Now we are ready to disprove Conjecture \ref{c}.

\begin{remark}\label{cl9}
If $G$ and $H$ are two $Cpo$-groups in which $|G|=55$, $|H|=22$, then by Theorem \ref{cl}, $|Cent(G)|=|Cent(H)|=13$  and by Part (2) of Theorem \ref{cl5} we have $|G'|=|H'|=11$ and $G, H$ are non-isoclinic groups.  
\end{remark}

\begin{cor}
 Remark \ref{cl9} holds for groups $|G|=qp$ and $|H|=rp$, where $ p, q, r$ are primes and $p>q, r$. Moreover $p-1=qr$.
\end{cor}

\begin{proof}
By applying Theorem \ref{cl} and Remark \ref{cl9}, the result follows.
\end{proof}

\begin{cor}

Let $G$ be a  non-abelian $Cpo$-group and $p$ is the smallest prime divisor of $\pi(G)$. Then $|Cent(G)|\ge p+3$.

\end{cor}
\begin{proof}
By Part (2) of Theorem \ref{cl5}, $|G|=pq$, where $p$ , $q$ are  distinct primes. Thus $|Cent(G)|=n_p+ n_q+1$  such that $n_p$, $n_q$  are the number of $p$-sylow subgroups and $q$-sylow subgroups of $G$ respectively. Besides since $n_p=ap+1$ for a natural number $a$ and $n_q=1$, Sylow's Theorem and Theorem \ref{cl4} imply that $|Cent(G)|=ap+3$, hence the result follows.
\end{proof}
\begin{lemma}
Let $P$ be a $p$-sylow subgroup and $M$ be a nilpotent subgroup of $G$. Then $|Cent( \frac{M}{M\cap P})|=|Cent(\frac{M}{M'\cap P})|$.
\end{lemma}

\begin{proof}
Assume that $P$ is a $p$-sylow subgroup and $M$ is a nilpotent subgroup of $G$, then $P\cap M$ is a $p$-sylow subgroup of $M$ and $P\cap M\unlhd M$. We put $K= M\cap P$, hence $K\cap M'= M'\cap P$. By  Lemma $2.2$ of \cite{B}, $|Cent( \frac{M}{K})|=|Cent(\frac{M}{K\cap M'})|=|Cent(\frac{M}{M'\cap P})|$.
\end{proof}

\begin{prop}
Let $G$ be a non-solvable group with a nilpotent maximal subgroup $C_G(x)$ for some $x\in G$. If a $2$-sylow subgroup of $C_G(x)$ is metacyclic and $G'\cap F(G)=1$, then $|Cent(G)|\ge 107$.
\end{prop}

\begin{proof}
As $G'\cap F(G)=1$, then $|Cent(G)|=|Cent(\frac{G}{F(G)})|$ by Lemma $2.2$ of \cite{B}. Besides from Theorem $4$ of \cite{Ba}, $\frac{G}{F(G)}$ is isomorphic to one of the following groups:\\
$PSL(2,p)$, $PGL(2, p)$, $PGL(2,7)$, $PGL(2,9)$, $H(9)$, where $p$ denotes a Fermat or Mersenne prime with $p\ge17$. Using the GAP  \cite{GAP} one can observe that $|Cent(PGL(2,7))|=107$ and this concludes the proof.
\end{proof}

\begin{lemma}
Let $G$ be a group and $G'\cap Z(G)=1$. If $\frac{G}{Z(G)}$  is perfect, then $|Cent(G)|=|Cent(G')|$.
\end{lemma}

\begin{proof}
Since $\frac{G}{Z(G)}$ is perfect and $G'\cap Z(G)=1$, so $G=G'\times Z(G)$. It implies that $|Cent(G)|=|Cent(G')|$.

\end{proof}

\end{document}